\documentclass{article}[12pt]

\usepackage{amsmath,amssymb,amsfonts,amsthm}
\usepackage[latin1]{inputenc}
\usepackage{eucal}
\usepackage{mathrsfs}   
\usepackage{dsfont}
\usepackage{amsmath,amssymb,amsfonts,amsthm}
\usepackage{mathtools} 
\usepackage{newlfont} 
\usepackage{fancyhdr,fancyvrb}
\RequirePackage[colorlinks,citecolor=blue,urlcolor=blue]{hyperref}
 
\newtheorem{theorem}{Theorem}[section]
\newtheorem{lemma}[theorem]{Lemma}
\newtheorem{proposition}[theorem]{Proposition} 
\newtheorem{corollary}[theorem]{Corollary}
\newtheorem{remark}[theorem]{Remark}

\def\RE{{\mathbb R}} 
\def\R{{\mathbb R}}
\def\Rm{{\mathbb R^m}}

\def\N{{\mathbb N}} 

\def\ud{\mathrm{d}}

\def\ind{\mathds{1}}

\begin{document}
 
\title{A counterexample to $L^{\infty}$-gradient type estimates for Ornstein-Uhlenbeck operators}     
 
\author{Emanuele Dolera
\\    
Enrico Priola 
\\ \\
   Dipartimento di  
 Matematica,  
 \\
Universit\`a  di Pavia, Pavia, Italy  
\\
emanuele.dolera@unipv.it\ \ enrico.priola@unipv.it
%\thanksmark{}
%University of Pavia \\
%Enrico Priola, University of Pavia 
}

\date{}

 \maketitle 
 
\abstract { Let $(\lambda_k)$ be a strictly  increasing sequence of positive numbers such that $\sum_{k=1}^{\infty} \frac{1}{\lambda_k} < \infty.$  Let $f $ be a bounded smooth function and  denote by $u= u^f$   the bounded classical solution to 
$u(x) - \frac{1}{2}\sum_{k=1}^m D^2_{kk} u(x) +  \sum_{k =1}^m \lambda_k x_k D_k u(x) = f(x), $ $ x \in \R^m$.
It is known that the following dimension-free estimate holds: 
$$
\displaystyle \int_{\R^m} \Big (\sum_{k=1}^m \lambda_k \, (D_k u (y))^2   \Big)^{p/2} \mu_m (dy)
 \le (c_p)^p  \, \int_{\R^m}  |f( y)|^p \mu_m (dy),\;\;\; 1 < p < \infty;
$$ 
here $\mu_m$ is the ``diagonal'' Gaussian measure determined by $\lambda_1, \ldots, \lambda_m$ and $c_p > 0$ is independent of $f$ and $m$. This is a consequence of generalized Meyer's inequalities \cite{goldysChojnowskaJFA}.  We show that, if $\lambda_k \sim k^2$, then   
such estimate does not hold when $p= \infty$. Indeed we prove  
$$
\sup_{\substack{f \in C^{ 2}_b(\R^m),\;\;  \|f\|_{\infty} \leq 1}} \Big \{ \sum_{k=1}^m \lambda_k \, (D_k u^f (0))^2 \Big \} \to \infty \;\; \text {as} \; m \to \infty.
$$
This is in contrast to the case of $\lambda_k = \lambda >0$, $k \ge 1$, where a dimension-free bound holds for $p =\infty$.}  
  
 \vspace{2.9 mm}


\vspace{2.7 mm}

 % Markov semigroups and applications to diffusion processes
 (60H15,
% Stochastic partial differential equations (aspects of stochastic analysis)
 42B37, 
 %(2010-now) Harmonic analysis and PDEs
 35R15) 
 %(1973-now) PDEs on infinite-dimensional

\section{Introduction and main result}

Let us recall dimension-free $L^p$-gradient estimates involving  Ornstein-Uhlenbeck operators (cf. \cite{meyer84, Shigekawa92, dapratoLincei, goldysChojnowskaJFA, ChGold02}). 
Let $(\lambda_k)$ be a   strictly increasing sequence of positive numbers such that 
\begin{equation}\label{trace} 
\sum_{k=1}^{\infty} \frac{1}{\lambda_k} < \infty.
\end{equation} 
For any $m \ge 1$ we denote by $A_m$ the  $m \times m$ diagonal matrix with negative eigenvalues $-\lambda_k$, $k=1, \ldots, m$.

Let $f : \R^m \to \R$ be a bounded   $C^2$-function with all first and second bounded derivatives, i.e., $f \in C^{ 2}_b(\R^m)$,  and denote by $u\in C^{ 2}_b(\R^m)$ the unique bounded classical solution to 
\begin{equation}\label{s2} 
u(x) - \Big ( \frac{1}{2} \triangle_m u(x)  + \langle A_mx , Du(x) \rangle  \Big) = u(x) - \frac{1}{2}\sum_{k=1}^m D^2_{kk} u(x) +  \sum_{k =1}^m \lambda_k x_k D_k u(x) = f(x), 
\end{equation}
where $x= (x_1, \ldots, x_m) \in \R^m$ and $\langle \cdot, \cdot\rangle$ denotes the standard scalar product in $\R^m$. See, for instance, \cite{DL}. Here, $D^2_{kk} $ and $D_k$ are first and second partial derivatives with respect to the canonical basis $(e_k)$ in $\R^m$. 
The operator we consider is an  $m$-dimensional Ornstein-Uhlenbeck operator, namely $L_m =  \frac{1}{2} \triangle_m  + \langle A_m x, D\rangle$.

Then, introduce the Gaussian measure $\mu_m= N(0, (-2 A_m)^{-1})$  with mean 0 and covariance matrix  $(-2A_m)^{-1}$, with density as in \eqref{gau1}.
Note that  $L_m$ is a self-adjoint operator on $L^2(\R^m, \mu_m)$ which is the usual $L^2$-space with respect to  $\mu_m$. See, for instance, \cite{gutierrez, dapratoLincei, goldysChojnowskaJFA, DZ1}.   
It is known that if $1< p< \infty$ there exists a  constant $c_p$ (independent of 
$f$ and the dimension $m$) such that  the following sharp gradient estimate holds: 
\begin{equation}\label{grad1}
\int_{\R^m} \Big (\sum_{k=1}^m \lambda_k \, (D_k u (y))^2   \Big)^{p/2} \mu_m (dy) \le (c_p)^p  \, \int_{\R^m}  |f( y)|^p \mu_m (dy).
\end{equation} 
The result follows by the general estimates \eqref{s223} given in  Theorem 5.3 of \cite{goldysChojnowskaJFA} which extends Proposition 3.5 in \cite{Shigekawa92} (see also the references therein). 
Note that  \eqref{grad1}  can be rewritten as  
\begin{equation}\label{s22}
 \| (-A_m)^{1/2} Du \|_{L^p(\R^m, \mu_m)}  \le c_p \|  f \|_{L^p(\R^m, \mu_m)},
\end{equation}
where $(-A_m)^{1/2} Du (x) = \sum_{k=1}^m \sqrt{\lambda_k} \, D_k u(x) e_k$.

Our  main result (cf. Theorem  \ref{main} below) shows that, when $p=\infty$, the dimension-free estimate \eqref{s22} in general fails to hold. Indeed, we prove the following  stronger assertion.  Writing $u =u^{f}$ to stress the dependence of the solution $u$ on $f$, we show  that if  $\lambda_k \sim k^2$ as $k \to \infty$, then, choosing $x=0$,  we have 
\begin{equation}\label{mm}
  \sup_{\substack{f \in C^{ 2}_b(\R^m) \\ \|f\|_{\infty} \leq 1}} |(-A_m)^{1/2} Du^f (0)|_{\R^m}^2 =
  \sup_{\substack{f \in C^{ 2}_b(\R^m) \\ \|f\|_{\infty} \leq 1}}  
   \Big \{ \sum_{k=1}^m \lambda_k \, (D_k u^f (0))^2 \Big \} \to \infty \;\; \text {as} \; m \to \infty.
\end{equation}
We point out that in contrast to \eqref{mm} when $A_m = - \lambda I_m$ with $\lambda >0$ and $I_m$ the $m \times m$ identity matrix then the dimension-free $L^{\infty}$-gradient estimates    
\begin{equation}\label{uno}
\|(\lambda)^{1/2} D u^f  \|_{\infty} = \sup_{x \in \R^m}\, |(\lambda)^{1/2} D u^f  (x)|_{\R^m} \, \le   \frac{\pi} {\sqrt 2}\,  \sup_{x \in \R^m}| f(x)|,\;\; f \in  C^{ 2}_b (\R^m)
\end{equation}
holds true; see Proposition \ref{wrr}.   
 
\vskip 1mm
 
Let us comment on the previous  dimension-free $L^p_{\mu}$-estimate \eqref{s22}. This  can be deduced by  known results for infinite dimensional Ornstein-Uhlenbeck operators. 
To introduce this setting, we replace 
$\R^m$  by a real separable Hilbert space $H$ with orthonormal basis $(e_k)_{k \ge 1} $ and inner product $\langle \cdot , \cdot \rangle$. Then, we consider the unbounded self-adjoint operator $A : D(A)\subset H\to H$ 
such that   
\begin{equation} \label{qsd}
 D(A) = \Big\{ x \in H \, :\, \sum_{k \ge 1} (\langle x,e_k \rangle)^2 \, \lambda_k^2 < \infty  \Big \},\;\;\;\; 
A e_k = - \lambda_k e_k, \;\; k \ge 1
\end{equation}
(cf. \cite{DZ, D2, bass, priolaAOP}). Our condition \eqref{trace} is equivalent to require that  the inverse operator $A^{-1} : H \to H$  is a trace class operator.
The operator $A$  generates a strongly continuous semigroup  $(e^{tA})$  on $H$, given by
$e^{tA} e_k = e^{-t\lambda_k} e_k$, $t \ge 0,$ $k \ge 1$. 
We can  define the corresponding  Ornstein-Uhlenbeck semigroup $(P_t)$:
\begin{equation} \label{ou}
P_t f (x)  = \int_{H} f(e^{tA} x+ \sqrt{I - e^{2tA}}\, y  )  \;     N  \big(0 , - (2 A)^{-1} \big ) \, (dy),\;\; f \in { B}_b (H),\; x \in H, \; t \ge 0
\end{equation}
where $f: H \to \R$ is a Borel, bounded function, and $N  \big(0 , - (2 A)^{-1} \big )$ stands for the centered Gaussian measure defined on the Borel $\sigma$-algebra of $H$ 
(see Chapter 1 in \cite{DZ1}, \cite{D2} and Section 2.2); $I$ is the identity.  
 
Formula \eqref{ou} is an extension of a well-known formula used in finite dimension.  
From the probabilistic point of view  $(P_t)$ is the transition Markov semigroup  of the OU stochastic process $(X_t^x)$  which solves $dX_t = AX_t dt +    dW_t,$ $ X_0 =x$ where $W$   is a {cylindrical Wiener} process on $H$; cf. \cite{DZ, DZ1, hairer}. When  $f\in C^2_b(H)$, i.e., $f$ is bounded, twice Fr\'echet-differentiable with first and second bounded and continuous derivatives, we consider   $u: H \to \R$,
\begin{equation} \label{uu}
 u(x) =  R(1, L) f(x) = \int_0^{\infty} e^{-t} (P_t f)(x)dt,\;\; x \in H.
\end{equation}
Following Chapter 6 in \cite{DZ1}, $u$ is the generalized bounded solution to $u - L u =f$, where $L$ is formally given by $  \frac{1}{2} \text{Tr} (D^2) + \langle x , AD\rangle $. Here, we only note that if $f$ is also  cylindrical, i.e., there exists $m \ge 1$  and $\tilde f \in C^2_b (\R^m)$ such that 
\begin{equation} \label{cil2}
 f(x) = \tilde f (\langle  x, e_{1}\rangle, \ldots, \langle  x, e_{m} \rangle),\;\;\; x \in H, 
\end{equation}   
then $u$ given in \eqref{uu}  depends only on  a finite number of variables, i.e., $u(x) = \tilde u (\langle  x, e_{1}\rangle, \ldots, \langle  x, e_{m} \rangle)$, $x \in H$ (cf. Section 2.2). Moreover,  
$\tilde u$ solves \eqref{s2} with $f$ replaced by $\tilde f$.
In addition, if  $f \in C_b^2(H)$, we have that $u = R(1, L) f \in C^2_b(H)$, and $Du(x) \in D((-A)^{1/2})$, $x \in H$ (cf. \cite{D2} for stronger results).   
 
By  Theorem 5.3 of \cite{goldysChojnowskaJFA} (see also Corollary 5.4 in \cite{goldysChojnowskaJFA} and Remark \ref{che}), there exists a  constant $c_p$ (independent of $f$) such that 
\begin{equation}\label{s223}
 \| (-A)^{1/2} Du \|_{L^p(H, \mu)}  \le c_p \|  f \|_{L^p(H, \mu)}, \;\;\; 1< p< \infty,
\end{equation} 
where   $\mu = N  \big(0 , - (2 A)^{-1} \big )$. Moreover, we have $ \| D^2 u \|_{L^p(H, \mu)}  \le c_p \|  f \|_{L^p(H, \mu)}$, i.e.,
\begin{equation}\label{22}
\int_{H} \Big (\sum_{k=1}^{\infty}  \, (D_{kk} u (y))^2   \Big)^{p/2} \mu (dy) \le (c_p)^p  \, \int_{H}  |f( y)|^p \mu (dy).
\end{equation} 
 It is  not difficult  to show  that \eqref{s223} implies \eqref{s22} using cylindrical functions $f$ as in  \eqref{cil2}; see  Section 2.2.

Estimates  \eqref{s223} and \eqref{22}  are part of the generalized  Meyer's inequalities proved in \cite{goldysChojnowskaJFA}  using 
the elliptic Littlewood-Paley-Stein inequalities associated with the OU semigroup $(P_t)$.  For applications of the 
classical    Meyer's inequalities  to the Malliavin Calculus we refer to \cite{meyer84, meyer, Nu} (see also  Remark \ref{mall}). 
The results given in \cite{goldysChojnowskaJFA} give a   characterization  of the domain of  the generator of $(P_t)$   in $L^p(H, \mu)$; see also \cite{ChGold02} (the case $p=2$  was  obtained  earlier in \cite{dapratoLincei}). We also mention the characterization   of the domain of non self-adjoint Ornstein-Uhlenbeck  generators given in \cite{lunardi, metafAltri, maasNeerven09}.
Estimates \eqref{22} have been  used to prove strong uniqueness for a class of SPDEs in \cite{DFPR}. For related results on Ornstein-Uhlenbeck operators in Gaussian harmonic analysis we refer to \cite{gutierrez, casar} and the references therein.

Our main result implies that \eqref{s223} fails to hold for $p= \infty$, i.e., it is not true that  
there exists $C >0$, independent of $f$, such that 
\begin{equation}\label{magari}         
\sup_{x \in H}\, | (-A)^{1/2} D R(1, L) f  (x)|_H \, \le C\,  \sup_{x \in H}| f(x)|,\;\; f \in C_b^2(H),
 \end{equation}
where  we have used  $u=  R(1, L) f$ as in \eqref{uu}. This estimate is stated in \cite[Theorem 7]{priolaAOP}
which is based on \cite[Lemma 6]{priolaAOP}. However, there is a mistake in the proof of such lemma. In particular  we show that \cite[Theorem 7]{priolaAOP} cannot hold.

\begin{remark}\label{che} {\em  Let us recall the notation used in \cite{goldysChojnowskaJFA}
to study general symmetric Ornstein-Uhlenbeck semigroups in Hilbert spaces.  
For the sake of notational clarity, the operator $C$ used in \cite{goldysChojnowskaJFA} corresponds to our $-(2 A)^{-1}$, while our semigroup $(e^{tA})$ corresponds to $(e^{-tA})$ in \cite{goldysChojnowskaJFA}. 
They use the Malliavin gradient $D_I = C^{1/2} D$ (where $D $ is the Fr\'echet derivative) and $D_A = \frac{1}{\sqrt{2}} D$. Moreover, the symbol
$D_{A^2} = A D_I$, which is used in the definition of the Sobolev space $W^{1,p}_{A^2}$  (see Corollary 5.4 in \cite{goldysChojnowskaJFA}) corresponds to our operator $\frac{{1}}{\sqrt{2}} (-A)^{1/2} D$.} 
\end{remark}

\begin{remark} \label{mall} {\em Let  us recall    the  classical Ornstein-Uhlenbeck semigroup $(S_t)$
\begin{equation}\label{ouu}
 S_t f(x)= \int_{H} f(e^{-t} x+ \sqrt{1 - e^{-2t}}y) \;      \nu \, (dy),\;\; f \in { B}_b (H),\; x \in H, 
\end{equation}
where $\nu$ is a centered  Gaussian measure on $H$ (see Section 2.2). The  classical  Meyer's inequalities give 
a complete characterization of the domains of $(I-N_p)^{m/2}$ in $L_p(H,\nu)$ for all  $p \in (1, \infty)$ and $m=1,2, \ldots$  in terms of Gaussian Sobolev spaces related to $\nu$. Here $N_p$ denotes the generator of $(S_t)$  in $L_p(H,\nu)$ (see \cite{meyer}, \cite{meyer84} and \cite{Nu}). }
%For a discussion of Meyer's inequalities in the Malliavin Calculus we refer to \cite{Nu}.} 
\end{remark}

\begin{remark} {\em Estimates like \eqref{s223} and \eqref{22} holds also in H\"older spaces (see \cite{D2, priolarxiv} for more details). In particular, for any $\theta \in (0,1)$, there exists an absolute constant only depending on $\theta$ such that 
\begin{equation}\label{33}
\| (-A)^{1/2} D R(1, L) f \|_{C^{\theta}_b(H, H)}  \le c_{\theta} \|  f \|_{ C^{\theta}_b(H)}.
\end{equation}
%see \cite{D2, priolarxiv} for more details.}
}
\end{remark}

\section{Notations and preliminary results} 
 
Let $Q$ be a  symmetric and positive definite $m \times m$ matrix; we denote by 
$N(0,Q)$ the Gaussian measure with mean 0 and covariance matrix $Q$;  it has density 
\begin{equation} \label{gau1}
(2 \pi)^{-m/2} (\det{Q})^{-1/2}      e^{- \frac{| (Q)^{-1/2} x|^2}{2}} 
\end{equation}
with respect to the $m$-dimensional Lebesgue measure. We first consider for $\lambda >0$ the equation
\begin{equation}\label{se}
v(x) - \Big ( \frac{1}{2} \triangle_m v(x)  -  \lambda \langle x , Dv(x) \rangle  \Big) = v(x) - M_m v(x) = f(x), \;\;\; x\in \R^m,
\end{equation}
with  $M_m = \frac{1}{2} \triangle_m - \lambda  \langle x, D \rangle$. We assume that $f \in C^2_b(\R^m)$. Equation \eqref{se}  is similar to \eqref{s2} with $A_m$ replaced by $-\lambda I_m$. Using the following Ornstein-Uhlenbeck semigroup $(S_t^m)$:
\begin{equation}\label{qss} 
 S_t^m f(x)=
\int_{\R^m} f(e^{-\lambda t} x+ \sqrt{1 - e^{-2 \lambda t}} \, y  )  \;      N\big(0, \frac{1}{2 \lambda}I_m \big) \, (dy),\;\;\; x \in \R^m, \; t \ge 0,
\end{equation} 
we find  (cf. \eqref{uu}, and \cite{DL,DZ1}) 
$$
v(x) = R(1, M_m) f(x) =  \int_0^{\infty} e^{-t} (S_t^m f)(x)dt,\;\; x \in \R^m.  
$$
Then, we have the following 
\begin{proposition} \label{wrr} 
For any $\lambda>0$ it holds:
\begin{equation}\label{ff}
 \sup_{x \in \R^m}\,  (\lambda)^{1/2} |D R(1, M_m) f (x)|_{\R^m} \, \le   \frac{\pi} {\sqrt 2}\,  \sup_{x \in \R^m}| f(x)|,\;\; f \in  C^{ 2}_b (\R^m). 
\end{equation}
 \end{proposition}
\begin{proof} 
Let $v(x) = R(1, M_m) f \in C^2_b(\R^m)$. We set $v(x) = u(\sqrt{\lambda} \, x)$ and so, for $y \in \R^m$, we get
$$
u(y) - \frac{\lambda }{2}\triangle u( y) +  \lambda  \langle y , Du(y) \rangle = f(y / \sqrt{\lambda})
$$
and
$$
\frac{1}{\lambda} u(y) -  \frac{1}{2}\triangle u( y) +  \langle   y , Du(y) \rangle  = \frac{1}{\lambda} f(y / \sqrt{\lambda}) =  \tilde f(y).
$$
We have 
$$
u(x) = \int_0^{\infty} e^{-\frac{1}{\lambda} \, t} dt \int_{\R^m} \tilde f(e^{-t } x + y )  \, N \Big(0, \frac{ 1-  e^{- 2t }} {2}  \,  I_m \Big) (dy)
$$
and,  considering the directional derivative $ \langle Du(x), h\rangle = D_h u(x)$, $h \in \R^m$,  $|h|=1$, we get  
$$
D_h u(x) = 2   \int_0^{\infty} e^{-\frac{1}{\lambda} \, t} \int_{\R^m} \tilde f(e^{-t } x + y )\frac{e^{- t} }{1 - e^{- 2 t } } \, \langle h,  y \rangle  \, N \Big(0, \frac{ 1-  e^{- 2t }} {2}  \,  I_m \Big) (dy)
$$
(cf. Theorem 6.2.2 in \cite{DZ1}, \cite{D2}  or page 101 in \cite{DL}). Then, changing variable in the integral over $\R^m$ and differentiating under the integral sign, we obtain
\begin{align*}
\| D_h u \|_{\infty} &\le  2 \| \tilde f\|_{\infty} \int_0^{\infty}   \frac{e^{-\frac{1}{\lambda}\, t}  \, e^{- t} }{1 - e^{- 2 t } } dt \int_{\R^m}   
\Big |\langle h, \big (\frac{ 1-  e^{- 2t }} {2}  \big )^{1/2}      y \rangle \Big |   \, N \big(0,  \,  I_m \big) (dy) \\ 
&\le  \frac{\sqrt{2}}{\lambda}    \| f\|_{\infty} \int_0^{\infty}   \frac{e^{- t}}{(1 - e^{- 2 t })^{1/2} } dt
\int_{\R^m} \big |\langle h, y \rangle \big | \, N \big(0,\, I_m \big) dy \le  \frac{\pi} {\lambda \sqrt 2} \| f\|_{\infty}.
\end{align*}  
Since $D_h u(y) = \frac{1}{\sqrt{\lambda}} D_h v(\frac{y} {\sqrt{\lambda}} )$  we have $\| D_h u \|_{\infty} = \frac{1}{\sqrt{\lambda}} \|  D_h v  \|_{\infty} $ and \eqref{ff} follows. 
\end{proof}
 
 \vskip 1mm 
 
Let us start the proof of the main estimate \eqref{mm} concerning equation \eqref{s2} involving 
 the Ornstein-Uhlenbeck operator
 $L_m$. Similarly to the proof of Proposition \ref{wrr}     
the  solution $u \in C^{ 2}_b(\R^m)$ to \eqref{s2} is given by
\begin{equation}\label{q55}
u(x)= R(1, L_m)f(x)= \int_0^{\infty} e^{-t} (P_t^m f)(x) dt
\end{equation}
with 
\begin{align*} 
P_t^m f(x) &=  \int_{\R^m} f(e^{tA_m} x+ \sqrt{I_m - e^{2tA_m}}\, y)\; N  \Big(0 , -\frac{1}{2} A^{-1}_m \Big ) \, (dy) \\ 
&=\int_{\R^m} f(e^{tA_m} x  +  y)\; N  \big(0 , Q_t^m  \big ) \, (dy), \;\;  f \in C^2_b(\R^m), \; x \in \R^m,
\end{align*}
where 
$$
Q_t^m = \int_0^t e^{2 sA_m} ds =  (-2 A_m)^{-1}(I_m -e^{2tA_m}),\;\;\; t \ge 0
$$
($Q_t^m$ is a diagonal matrix with positive eigenvalues). Let $\mu_t^m= N  \big(0 , Q_t^m  \big )$. The following formula holds for the directional derivative of $P_t^m f$ along   $h \in \R^m$: 
\begin{equation}\label{e3} 
D_h P_t^m f(x) =  
\langle D P_t^m f(x),h
 \rangle  = \int_{\R^m} \langle
  \Lambda_t^m h,(Q_t^m)^{-\frac12} y\rangle \, f (e^{tA_m}x+y ) \mu_t^m(dy), 
  \; x \in \R^m, \; t>0,
\end{equation}
where $ \Lambda_t^m = (Q_t^m)^{-1/2}e^{tA_m};$ cf. Theorem 6.2.2 in \cite{DZ1}  or page 101 in \cite{DL}. Hence  
$$
(-A_m)^{1/2} Du^f (0) =   (-A_m)^{1/2}  D R(1, L_m)f (0) \in \R^m
$$ 
appearing in \eqref{mm} has components  %(we use $\langle \cdot, \cdot \rangle = \langle \cdot, \cdot \rangle_{\R^m}$ and $|\cdot | = |\cdot |_{\R^m}$)
\begin{align*}
\langle  (-A_m)^{1/2} Du^f (0),e_k\rangle &=  \int_0^{\infty} e^{-t} dt \int_{\R^m} \langle
(-A_m)^{1/2} \Lambda_t^m e_k,(Q_t^m)^{-\frac12} y\rangle \, f ( y ) \mu_t^m(dy)\\
&=  \int_0^{\infty} e^{-t} dt \int_{\R^m} \langle (-A_m)^{1/2} \Lambda_t^m e_k,y\rangle \,  f ((Q_t^m)^{\frac12}  y ) \,N(0,I_m)(dy) ,\;\; k = 1, \ldots, m.
\end{align*}  
An easy calculation shows that 
\begin{align} \label{si1}
&|(-A_m)^{1/2} D  R(1, L_m)f  (0)|^{2}\\  
&=  \sum_{k=1}^m  \Big(\int_0^{\infty}\dfrac{\lambda_k  e^{-t} e^{-\lambda_k t}}{(1 - e^{-2\lambda_k t})^{1/2}}   \frac{1}{\sqrt{(2 \pi)^m} } \int_{\R^m}  f (c_1 (t) x_1, \ldots,  c_m(t)x_m)\,  x_k \, e^{- \frac{|x|^2}{2}} dx  
dt \Big)^2,  
\end{align} 
where, for $k \in \{1, \dots, m\}$ and $t\ge 0$, $c_k(t) = \left(\dfrac{1 - e^{-2\lambda_k t}}{2\lambda_k}\right)^{1/2}$ and $(Q_t^m)^{1/2} = \text{diag}[c_1(t), \dots, c_m(t)]$.

We will prove the following result.
\begin{theorem}\label{main} Let $(\lambda_k)$ be a strictly increasing sequence of positive numbers, such that 
$\lambda_k \sim k^2$ as $k \to +\infty$. Then, assertion \eqref{mm} is in force, i.e., taking into account \eqref{si1}, there holds  
$$
\sup_{m \in \N} \ \sup_{\substack{f \in  C^{ 2}_b (\R^m) \\ \|f\|_{\infty} \leq 1}} \sum_{k=1}^m \left(\int_0^{\infty}\dfrac{\lambda_k  e^{-t} e^{-\lambda_k t}}{(1 - e^{-2\lambda_k t})^{1/2}} 
\frac{1}{\sqrt{(2 \pi)^m} } \int_{\R^m}  f (c_1 (t) x_1, \ldots, c_m(t)x_m)\,  x_k \, e^{- \frac{|x|^2}{2}} dx dt  \right)^2 = +\infty.
$$ 
\end{theorem}
The proof of the theorem is given in Section 3. Next, we discuss an application of Theorem \ref{main} to infinite dimensions, see Corollary \ref{srr}.

\subsection{An infinite dimensional  Ornstein-Uhlenbeck semigroup}

Let $H$ be a real separable Hilbert space with inner product $\langle \cdot , \cdot \rangle$. Let $Q: H \to H$ be a symmetric non-negative definite trace class operator. The centered Gaussian measure  $\mu = N(0,Q)$  is the unique  probability measure on the Borel $\sigma$-algebra of $H$ such that 
\begin{equation}\label{fu}
 \int_H e^{i \langle x,h\rangle } \mu (dx) = e^{-\frac{1}{2} \langle Qh,h \rangle },\;\; h \in H  
\end{equation}
(cf. \cite{DFPR}).  
 We denote by $B_b(H)$ the Banach space of all Borel and bounded real functions  endowed with the supremum norm $ \| \cdot \|_{\infty}$.  Moreover,  $C^{2}_b (H) \subset {B}_b(H)$ is the  space of all functions  which are bounded and Fr\'echet differentiable on $H$ up to the  second order  with all the derivatives $D^j f$ bounded and continuous on $H$, $ j =1,  2$. 
According to Chapter  1 in \cite{DZ1} we can rewrite the OU semigroup $(P_t)$ in \eqref{ou} as follows
\begin{equation} \label{ou3}
P_t f (x)  = \int_{H} f(e^{tA} x+  y)\;  N  \big(0 ,  Q_t \big ) \, (dy),\;\; f \in { B}_b (H),\; x \in H, 
\end{equation}
where $Q_t = \int_0^t e^{2 sA} ds =  (-2 A)^{-1}(I-e^{2tA}),\; t \ge 0,$ and $A$ is given in \eqref{qsd}.
Suppose that   $f \in C^2_b(H)$ is  also {\sl  cylindrical,} i.e., there exists $m \ge 1$  and $\tilde f \in C_b^2 (\R^m)$ such that \eqref{cil2} holds. This is equivalent to require that $f = f \circ \pi_m$, using the 
finite dimensional approximations   $\pi_m=\sum_{j=1}^me_j\otimes e_j$.

Identifying $H $ with $l^2$, we have: $f(e^{tA} x + y) = \tilde f (e^{tA_m}x^{(m)} + y^{(m)})$ using the notation
$$
h^{(m)} = (\langle  h, e_{1}\rangle, \ldots, \langle  h, e_{m} \rangle) \in \R^m,\;\; \text{for any} \; h \in H\ ,
$$
while $A_m$ is the same matrix given in \eqref{s2} and \eqref{q55}. 
Moreover, $N  \big(0 , - (2 A)^{-1} \big ) = N  \big(0 , - (2 A_m)^{-1} \big ) \times \nu_m $ where $\nu_m = \prod_{k= m+1}^{\infty} N(0,  (2 \lambda_k)^{-1} )$; see Theorem 1.2.1 in  \cite{DZ1}. It follows that,  for any $x \in H$,
\begin{gather*}
P_t f (x)=  P_t^m (\tilde f)( x^{(m)}) =
 \int_{\R^m } \tilde f( e^{tA_m} x^{(m)} + \sqrt{I_m - e^{2tA_m}}\,  y  )  \;     N  \big(0 , - (2 A_m)^{-1} \big ) \, (dy); 
\\    
u(x) =  R(1,L)f(x) =
   \int_0^{\infty} e^{-t} (P_t f)(x)dt = \tilde u (\langle  x, e_{1}\rangle, \ldots, \langle  x, e_{m} \rangle);  
  \\ 
  \tilde u (z) = \int_0^{\infty} e^{-t} P_t^m \tilde f(z)
  ,\;\; z \in \R^m.
\end{gather*}
Recall that  $P_t^m $ is given in \eqref{q55}. 
Setting    $\mu_m =  N  \big(0 , - (2 A_m)^{-1} \big )$ and using that $C^2_b(H)$ contains in particular cylindrical functions as in \eqref{cil2} we infer, for any $m \ge 1$,   
\begin{equation}\label{qww}
\sup_{\substack{\tilde f \in C^{ 2}_b(\R^m) \\ \| \tilde f\|_{L^p(\R^m, \mu_m)} \leq 1}} \,  \| (-A_m)^{1/2} D \tilde u \|_{L^p(\R^m, \mu_m)}
\le {\sup_{\substack{f \in C^{ 2}_b(H) \\ \|f\|_{L^p(H, \mu)} \leq 1}}} \| (-A)^{1/2} Du \|_{L^p(H, \mu)},\;\;\;\; 1<p< \infty, 
\end{equation}
and 
\begin{equation}\label{qww1}
\sup_{\substack{ \tilde f \in C^{ 2}_b(\R^m),\\ \|\tilde f\|_{\infty} \leq 1}} \,  \| (-A_m)^{1/2} D \tilde u \|_{\infty} \le \sup_{\substack{f \in C^{ 2}_b(H), \\  \|f\|_{\infty} \leq 1}} \, \sup_{x \in H } \, | (-A)^{1/2} Du (x)|_{H}.     
\end{equation}    
As a consequence of Theorem \ref{main} we obtain (see \eqref{uu})
\begin{corollary} \label{srr}
Under the same assumptions of Theorem \ref{main}, there holds 
$$   
\sup_{\substack{f \in C_b^{2}(H),\;\; \|f\|_{\infty} \leq 1}} | (-A)^{1/2} D (R(1, L) f)\, (0)|_H = \infty.
$$
\end{corollary}

\section{ Proof of Theorem \ref{main}  }

Let $\delta \in (0,+\infty)$. Then, put
\begin{equation*} 
 S_m = S_m (\delta) =  
\ \sup_{\substack{f \in  C^{ 2}_b (\R^m) \\ \|f\|_{\infty} \leq 1}} \sum_{k=1}^m \Big(\int_0^{\delta} \dfrac{\lambda_k e^{-\lambda_k t}}{(1 - e^{-2\lambda_k t})^{1/2}}    \int_{\R^m}  f (c_1 (t) x_1, \ldots, c_m(t)x_m)\,  x_k \, \frac{e^{- \frac{|x|^2}{2}}}{{\sqrt{(2 \pi)^m} } }  dx dt \Big)^2\ .  
\end{equation*}
If we show that 
\begin{equation} \label{main0} 
 \sup_{m \ge 2} S_m = \infty
\end{equation}
holds under the assumption that $\lambda_k \sim k^2$ as $k \to +\infty$, then the validity of Theorem \ref{main} will follow.

\subsection{Two useful  lemmas} 

The following identity will be important. Recall that $x_k = \langle x, e_k \rangle $, $k =1, \ldots, m$ where $(e_j)$ denotes the canonical basis in $\R^m. $
\begin{lemma}\label{aux} 
For any $m \geq 2$, $k  \in \{1, \dots, m\}$, $ c = (c_1, \dots, c_m) \in \Rm\setminus\{0\}$ and $F \in B_b(\RE)$, it holds 
\begin{align} \label{first}
I_{m,k}(F) &= \frac{1}{\sqrt{(2 \pi)^m} } \int_{\R^m}  F(\langle c ,  x \rangle ) x_k \, e^{- \frac{|x|^2}{2}} dx  \\ \nonumber 
&= \dfrac{2\pi (\sqrt{\pi})^{m-3}}{(2\pi)^{m/2} \Gamma\left(\frac{m-1}{2}\right)} \frac{c_k}{| c|}\ \int_0^{+\infty}\!\!\!\int_0^{\pi} e^{-\frac 12\rho^2} \rho^m
\cos\vartheta (\sin\vartheta)^{m-2} F(| c| \rho\cos\vartheta) \ud\rho\ud \vartheta.
\end{align}   
\end{lemma} 

\begin{proof}  We provide additional details for the sake of completeness. 
Let us first consider $m=2$. We introduce the unitary vectors $\gamma_1 = c/|c|$ and $\gamma_2 \in \R^2$ such that $(\gamma_1, \gamma_2)$ is an orthonormal basis in $\R^2$. Using the polar coordinates with respect to such basis we can write
$$
x = \rho \cos \theta \, \gamma_1 + \rho \sin \theta \, \gamma_2, 
$$
which entails that
\begin{align*}
I_{2,k}(F)&= \frac{1}{{2 \pi} } \int_0^{2\pi}  \int_{\R} \rho^2  F(|c| \rho \cos \theta ) \, ( \cos \theta \, \langle \gamma_1, e_k \rangle  +  \sin  \theta \, \langle \gamma_2 , e_k \rangle) \, e^{-\frac{\rho^2}{2}} d \rho d \theta\\
&= \frac{1}{{ \pi} } \int_0^{\pi}  \int_{\R} \rho^2  F(|c| \rho \cos \theta ) \, \cos \theta \, \langle \gamma_1 , e_k \rangle  \, e^{-\frac{\rho^2}{2}} d \rho d \theta,\;\;\; k =1,2
\end{align*}
since $\int_0^{2\pi} F(|c| \rho \cos \theta )\, \sin\theta\, d \theta =0$. We get easily \eqref{first} for $m=2$ recalling that $\Gamma (1/2) = \sqrt{\pi}$.

In the general case of $m \ge 3$, we consider an orthonormal basis $(\gamma_k)$ of $\R^m$ where $\gamma_1 = c/|c|$. Then, we introduce polar coordinates with respect to $(\gamma_k)$. Let $\rho = |x|$. Proceeding similarly to \cite[Section 5.9]{fleming}, we have, for $x \not = 0$, 
$$
 x = \rho \cos \theta_1 \gamma_1 +  \rho \sin \theta_1\cos \theta_2   \gamma_2
 + \ldots +  \rho \sin \theta_1 \cdots \sin \theta_{m-2}  \sin \theta_{m-1} \gamma_m, 
 $$
where $\rho >0$ (radial distance), $\theta_1 , \ldots , \theta_{m-2} \in [0, \pi]$ (latitudes; $\theta_1$ is the angle between $x$ and $\gamma_1$) and $\theta_{m-1} \in [0, 2 \pi]$ (longitude). 
Let $\theta = (\theta_1, \ldots, \theta_{m-1})$. Denote by  
$$
J(\rho, \theta) = \rho^{m-1} (\sin \theta_1)^{m-2} (\sin \theta_2)^{m-3} \cdots  (\sin \theta_{m-2})
$$    
the Jacobian determinant. Moreover, set $\gamma_i^{(k)} =\langle \gamma_i, e_k \rangle$, for $i,k =1, \ldots, m$. Let 
 \begin{gather*}
 \xi_1 (\theta) =  \cos \theta_1 , \;\;  \xi_2 (\theta) =   \sin \theta_1\cos \theta_2, \; \ldots,   
 \\ \xi_{m-1} (\theta) =  \sin \theta_1 \cdots \sin \theta_{m-2}  \cos \theta_{m-1},\;\;\;  \xi_{m} (\theta) =  \sin \theta_1 \cdots \sin \theta_{m-2}  \sin \theta_{m-1}.
\end{gather*}
For instance, for $m=4$, we have: $\xi_1 (\theta) =  \cos \theta_1 , \;\;  \xi_2 (\theta) =   \sin \theta_1\cos \theta_2$, $\xi_3(\theta) = \sin \theta_1 \sin \theta_2 \cos \theta_3$, 
$\xi_4(\theta)= \sin \theta_1\sin \theta_2  \sin \theta_3$, with $\theta_1, \theta_2 \in [0, \pi]$ and $\theta_3 \in [0, 2 \pi]$. 
We infer that
\begin{align} \label{ok} 
I_{m,k}(F) &= \frac{1}{\sqrt{(2 \pi)^m} } \int_0^{\infty} \int_{[0, \pi]^{m-2} \times [0, 2\pi]}  \rho e^{-\frac{\rho^2}{2}} F(|c|\, \rho \cos \theta_1) \, \left(\sum_{i=1}^{m} \xi_i(\theta) \gamma_i^{(k)} \right)
J(\rho, \theta) d \rho d \theta \nonumber \\
&=  \frac{1}{\sqrt{(2 \pi)^m} } \int_0^{\infty} \int_{[0, \pi]^{m-2} \times [0, 2\pi]}  \rho e^{-\frac{\rho^2}{2}} F(|c|\, \rho \cos \theta_1) \, \xi_1(\theta) \gamma_1^{(k)} \, J(\rho, \theta) d \rho d \theta
\end{align}
using that 
\begin{equation} \label{prova}
\frac{1}{\sqrt{(2 \pi)^m} } \int_0^{\infty} \int_{[0, \pi]^{m-2} \times [0, 2\pi]}  \rho e^{-\frac{\rho^2}{2}} F(|c|\, \rho \cos \theta_1) \, \left(\sum_{i=2}^{m} \xi_i(\theta) \gamma_i^{(k)} \right) J(\rho, \theta) d \rho d \theta =0.
\end{equation}
In order to prove  \eqref{prova} we check that if   $\rho>0$ then
\begin{equation}\label{sd}
\int_{[0, \pi]^{m-2} \times [0, 2\pi]}    F(|c|\, \rho \cos \theta_1) \, \xi_i(\theta)  J(\rho, \theta) d \theta =0,\;\;\; 2 \le i \le m.
\end{equation}
If $i = m$, we find that   
\begin{align*}
&\int_{[0, \pi]^{m-2} \times [0, 2\pi]}  \,  F(|c|\, \rho \cos \theta_1)\,  \xi_m(\theta)\,  (\sin \theta_1)^{m-2}  (\sin \theta_2)^{m-3}  \cdots  (\sin \theta_{m-2}) d\theta \\
&=    \int_0^{\pi} F(|c|\, \rho \cos \theta_1)  (\sin \theta_1)^{m-2} \sin \theta_1 d \theta_1 \times \\
&\times  \int_{[0, \pi]^{m-3} \times [0, 2\pi]}   \,  \sin \theta_2 \cdots       \sin \theta_{m-1}\,   (\sin \theta_2)^{m-3}   \cdots  (\sin \theta_{m-2}) d\theta_2 \cdots d \theta_{m-1} =0
\end{align*}
by the Fubini theorem, since $\int_0^{2 \pi}  \sin \theta_{m-1} d \theta_{m-1} =0.$ Similarly we obtain that \eqref{sd} holds with $i = m-1$. 
Note that  up to now we have already proved \eqref{sd} when $m=3$. Let $m \ge 4$. We check   \eqref{sd}   when $2  < i  \le m-2$. We have
\begin{align*}
&\int_{[0, \pi]^{m-2} \times [0, 2\pi]}  \,  F(|c|\, \rho \cos \theta_1)\, \xi_i(\theta)\, (\sin \theta_1)^{m-2}  (\sin \theta_2)^{m-3}  \cdots  (\sin \theta_{m-2}) d\theta \\
&= \int_0^{\pi} F(|c|\, \rho \cos \theta_1)  (\sin \theta_1)^{m-2} \sin \theta_1 d \theta_1 \times \\
&\times\int_{[0, \pi]^{m-3} \times [0, 2\pi]}   \,  \sin \theta_2 \cdots \cos \theta_{i}\, (\sin \theta_2)^{m-3}   \cdots  (\sin \theta_{m-2}) d\theta_2 \cdots d \theta_{m-1} =0, 
\end{align*} 
because  $\int_0^{ \pi}  \cos  \theta_{i} \,  (\sin \theta_{i})^{m-1 -i} d \theta_{i} =0.$ Similarly, for $i=2$, we get
\begin{align*}
& \int_0^{\pi} F(|c|\, \rho \cos \theta_1)  (\sin \theta_1)^{m-2} \sin \theta_1 d \theta_1 \times \\ 
& \times \int_{[0, \pi]^{m-3} \times [0, 2\pi]}\, \cos \theta_{2}\, (\sin \theta_2)^{m-3}   \cdots  (\sin \theta_{m-2}) d\theta_2 \cdots d \theta_{m-1} =0.  
\end{align*} 
We have verified   \eqref{sd} and so \eqref{ok} holds.  We rewrite \eqref{ok} as follow 
\begin{gather} \label{sw1} 
I_{m,k}(F) =  R_m \, 
\frac{  \gamma_1^{(k)}  }{\sqrt{(2 \pi)^m} } \int_0^{\infty} \int_{0}^{\pi}
\rho^m e^{-\frac{\rho^2}{2}} F(|c|\, \rho \cos \theta_1) \, 
   \cos \theta_1\,  (\sin \theta_1)^{m-2}  d \rho d \theta_1,\;\;\;  \gamma_1^{(k)}
   = \frac{c_k}{|c|},    
\end{gather}
 where  $R_m = 2 \pi$ if $m=3$ and if $m>3$ 
 \begin{align*} 
R_m &= \int_{[0, \pi]^{m-3} \times [0, 2\pi]}  \, (\sin \theta_2)^{m-3}    (\sin \theta_3)^{m-4} \cdots  \sin \theta_{m-2} d\theta_2 \cdots d \theta_{m-1}\\  
&= 2 \pi \prod_{j=1}^{m-3} \int_0^{\pi} (\sin \phi)^j d \phi = 2 \pi  \prod_{j=1}^{m-3} B\left(\frac{j+1}{2}, \frac{1}{2} \right) = 2 \pi  \prod_{j=1}^{m-3} 
\frac{ \Gamma\left( \frac{j+1}{2}\right) \Gamma(\frac{1}{2})}{\Gamma\left(\frac{j+2}{2} \right)}.    
\end{align*}
  We have used the Beta function $B(\cdot, \cdot)$ (cf.   page 103  of \cite{spiegel}).   Hence since $\Gamma(1/2) = \sqrt{\pi}$, we get
  \begin{gather*}\label{rm}  
 R_m =  
  2 \pi  (\sqrt{\pi})^{m-3}   \Big ( {\Gamma  \big( \frac{m-1}{2} \big)} \Big)^{-1}.    
\end{gather*} 
Inserting $R_m$ in \eqref{sw1} we obtain \eqref{first}, i.e., 
\begin{gather*} 
 I_{m,k} (F)=
%\frac{1}{\sqrt{(2 \pi)^m} } \int_{\R^m}  F(c \cdot x) x_k \, e^{-\frac{|x|^2}{2}} dx 
%\\ \nonumber =  
  \dfrac{2\pi (\sqrt{\pi})^{m-3}}{(2\pi)^{m/2} \Gamma\left(\frac{m-1}{2}\right)} \frac{c_k}{| c|}\ \int_0^{+\infty}\!\!\!\int_0^{\pi} e^{-\frac 12\rho^2} \rho^m
\cos\vartheta (\sin\vartheta)^{m-2} F(| c| \rho\cos\vartheta) \ud\rho\ud \vartheta.
\end{gather*}
\end{proof}

\begin{lemma} \label{due} 
If $F \in B_b(\R)$ verifies $F(x) = - F(-x)$ for any $x \in \R$, then we have, 
for any $m \geq 2$, $k  \in \{1, \dots, m\}$, $ c = (c_1, \dots, c_m) \in \Rm\setminus\{0\}$,
\begin{equation}\label{aqw}
I_{m,k} (F)
= 
\dfrac{  4\pi (\sqrt{\pi})^{m-3}}{(2\pi)^{m/2} \Gamma\left(\frac{m-1}{2}\right)} \frac{c_k}{| c|}\ \int_0^{+\infty}\!\!\! e^{-\frac 12\rho^2} \rho^m \ud\rho  
\int_0^1 x (1-x^2)^{\frac{m-3}{2}} F(|c| \rho x) \ud x 
\end{equation}
(cf. \eqref{first}). In the special case of   $F= F_0 := \ind_{(0, \infty)} - \ind_{(- \infty, 0)}$, we obtain 
\begin{gather} \label{ciao} 
 I_{m,k}(F_0) = 
  \frac{1}{\sqrt{(2 \pi)^m} } \int_{\R^m}  F_0(\langle c, x\rangle) x_k \, e^{- \frac{|x|^2}{2}} dx  = \frac{\sqrt{2}}{\sqrt{\pi}} \frac{c_k}{| c|}.
 \end{gather}
\end{lemma}
\begin{proof}  By changing variable $x= \cos \theta$ and using that $F(x) = - F(-x)$, $x \not =0$, we have  
$$ 
\int_0^{\pi} \cos\vartheta (\sin\vartheta)^{m-2} F(|c| \rho\cos\vartheta)\ud\vartheta = 2 \int_{0}^1 x (1-x^2)^{\frac{m-3}{2}} F(|c| \rho x) dx\ .
$$ 
Whence,
$$
I_{m,k} (F) = \dfrac{2 \cdot 2\pi (\sqrt{\pi})^{m-3}}{(2\pi)^{m/2} \Gamma\left(\frac{m-1}{2}\right)} \frac{c_k}{| c|}\ \int_0^{+\infty}\!\!\! e^{-\frac 12\rho^2} \rho^m d\rho  
\int_0^1 x (1-x^2)^{\frac{m-3}{2}} F(|c| \rho x) dx.
$$
Let us assume that $F =F_0 =\ind_{(0, \infty)} -   \ind_{(- \infty, 0)} $. We find 
$$
I_{m,k} (F_0) = \dfrac{4\pi (\sqrt{\pi})^{m-3}}{(2\pi)^{m/2} \Gamma\left(\frac{m-1}{2}\right)} \frac{c_k}{| c|}\ \int_0^{+\infty}\!\!\! e^{-\frac 12\rho^2} \rho^m d\rho  
\int_0^1 x (1-x^2)^{\frac{m-3}{2}} dx.
$$
Using that $\int_0^1 x (1-x^2)^{\frac{m-3}{2}}   \ud x  = \frac{1}{m-1}$ and 
$$
\int_0^{+\infty}\!\!\! e^{-\frac 12\rho^2} \rho^m \ud\rho = \Gamma\left(\frac{m+1}{2}\right) 2^{\frac{m-1}{2}}
$$
we find 
\begin{align*} 
I_{m,k}(F_0) &=  \dfrac{4\pi (\sqrt{\pi})^{m-3}}{(2\pi)^{m/2} \Gamma\left(\frac{m-1}{2}\right)} \Gamma\left(\frac{m+1}{2}\right) 2^{\frac{m-1}{2}} \,  \frac{1}{m-1} \,  \frac{c_k}{| c|}\\
&= \dfrac{4\pi (\sqrt{\pi})^{m-3}}{(\pi)^{m/2}  } \left(\frac{m-1}{2}\right) 2^{-\frac{1}{2}} \,  \frac{1}{m-1} \,  \frac{c_k}{| c|} = \frac{\sqrt{2}}{\sqrt{\pi}} \frac{c_k}{| c|},
\end{align*}
since $x \Gamma (x) = \Gamma (x+1)$, $x>0$  and this finishes the proof.
\end{proof}

\subsection{Proof of assertion \eqref{main0}}

Recall that $c_k(t) := \left(\dfrac{1 - e^{-2\lambda_k t}}{2\lambda_k}\right)^{1/2}$ for $k \in \{1, \dots, m\}$ and $t \ge 0$. Set 
$c(t) = (c_1(t), \dots, c_m(t)) \in \R^m$. 
Fix $m \ge 2$ and $\delta >0$, and put $ S_m := S_m(\delta)$. Then, for $m \ge 2$, define
\begin{equation} \label{aaa}
 A_m := \frac{{2}}{{\pi}}  \sum_{k=1}^m  \left(\int_0^{\delta} \dfrac{\lambda_k e^{-\lambda_k t}}{(1 - e^{-2\lambda_k t})^{1/2}}\, \frac{c_k(t)}{|c(t)|} \, dt\right)^2.
\end{equation} 
We prove that $\lim_{m \to \infty} S_m = \infty$ in two steps. \\

\noindent {\it I step.} We prove 
\begin{equation}\label{sfg}
 S_m \ge A_m,\;\;\; \forall\ m \ge 2.
\end{equation}
We start by constructing an approximating sequence of smooth functions for $F_0 := \ind_{(0, \infty)} -   \ind_{(- \infty, 0)}$.
For any $n \ge 1$, consider a non-decreasing $F_n  \in C^{ 2}_b (\R_+) $ such that $F_n (y) =0$ if $0 \le y \le 1/(n+1)$ and $F_n (y) = 1$ if $y \ge 1/n$.
Then, extend each $F_n$ to an odd function on $\R$ by the rule $F_n(x) = - F_n(-x)$ if $x<0$, and define
\begin{gather*}
 f_n (x_1, \ldots, x_m) = F_n  (x_1 + \ldots + x_m), \;\;\; x_1, \ldots, x_m \in \R.
\end{gather*}
It is clear that each $f_n \in C^{ 2}_b(\R^m)$ and $\| f_n\|_{\infty} \le 1$. Whence, 
\begin{align*}
S_m &\ge  \sup_{n \ge 1}\sum_{k=1}^m \Big(\int_0^{\delta} \dfrac{\lambda_k e^{-\lambda_k t}}{(1 - e^{-2\lambda_k t})^{1/2}}  \frac{1}{\sqrt{(2 \pi)^m} } \int_{\R^m}  
 f_n (c_1 (t) x_1, \ldots,  c_m(t)x_m)\,  x_k  \, e^{- \frac{|x|^2}{2}} dx dt \Big)^2 \\
&= \sup_{n \ge 1}  \sum_{k=1}^m \Big(\int_0^{\delta} \dfrac{\lambda_k e^{-\lambda_k t}}{(1 - e^{-2\lambda_k t})^{1/2}}  \frac{1}{\sqrt{(2 \pi)^m} } \int_{\R^m}  F_n(\langle c(t) , x \rangle ) x_k \, e^{- \frac{|x|^2}{2}} dx 
dt \Big)^2.
\end{align*}
Moreover, combining the fact that each $F_n$ is an odd functions with \eqref{aqw}, with $c$ replaced by $c(t)$, yields
 \begin{align*}
&\sup_{n \ge 1} \sum_{k=1}^m \Big(\int_0^{\delta} \dfrac{\lambda_k e^{-\lambda_k t}}{(1 - e^{-2\lambda_k t})^{1/2}}  \frac{1}{\sqrt{(2 \pi)^m} } \int_{\R^m}  F_n(\langle c(t) , x \rangle) x_k \, e^{- \frac{|x|^2}{2}} dx 
dt \Big)^2 \\ 
&=  \sup_{n \ge 1} \, 
\dfrac{  4\pi (\sqrt{\pi})^{m-3}}{(2\pi)^{m/2} \Gamma\left(\frac{m-1}{2}\right)}
 \sum_{k=1}^m \Big( \frac{c_k(t)}{| c(t)|}\ \int_0^{+\infty}\!\!\! e^{-\frac 12\rho^2} \rho^m d\rho  
\int_0^1 x (1-x^2)^{\frac{m-3}{2}} F_n(|c(t)| \rho x) dx \Big)^2.
\end{align*}
Then, using that both $F_n(x) \le F_{n+1}(x)$ and $F_n(x) \to F_0(x)$ hold for any $x \ge 0$, apply the monotone convergence theorem to get 
 \begin{align*}
S_m &\ge \sup_{n \ge 1} 
\dfrac{  4\pi (\sqrt{\pi})^{m-3}}{(2\pi)^{m/2} \Gamma\left(\frac{m-1}{2}\right)}   \sum_{k=1}^m \Big( \frac{c_k(t)}{| c(t)|}\ \int_0^{+\infty}\!\!\! e^{-\frac 12\rho^2} \rho^m d\rho  
\int_0^1 x (1-x^2)^{\frac{m-3}{2}} F_n(|c(t)| \rho x) \ud x \Big)^2\\ 
&=\dfrac{  4\pi (\sqrt{\pi})^{m-3}}{(2\pi)^{m/2} \Gamma\left(\frac{m-1}{2}\right)}  \sum_{k=1}^m \Big( \frac{c_k(t)}{| c(t)|}\ \int_0^{+\infty}\!\!\! e^{-\frac 12\rho^2} \rho^m d\rho  
\int_0^1 x (1-x^2)^{\frac{m-3}{2}} F_0(|c(t)| \rho x) \ud x \Big)^2\\
&=  \sum_{k=1}^m \Big(\int_0^{\delta} \dfrac{\lambda_k e^{-\lambda_k t}}{(1 - e^{-2\lambda_k t})^{1/2}}  \frac{1}{\sqrt{(2 \pi)^m} } \int_{\R^m}  F_0(\langle 
c(t),   x \rangle ) x_k \, e^{- \frac{|x|^2}{2}} dx  
dt \Big)^2 = A_m,  
\end{align*}
for $m \ge 2$. In the last  line  we have used both  \eqref{aqw} and  \eqref{ciao} with $c$ replaced by $c(t)$. This proves \eqref{sfg}.  \\

\noindent  {\it II step.} We prove that
\begin{equation}\label{qee}   
 \lim_{m \to \infty} A_m = \infty\ ,
\end{equation}
thus completing the proof of \eqref{main0}. Recalling the definition of $c_k(t)$, we have  
\begin{equation} \label{sqq}
 A_m = \frac{{2}}{{\pi}}  \sum_{k=1}^m \Big(  \int_0^{\delta} \dfrac{\sqrt{\lambda_k}  e^{-\lambda_k t}}{\sqrt{2}}\, \frac{1}{|c(t)|}\, dt\Big)^2  \ge 
 \frac{{1}}{{\pi}}\sum_{k=1}^m    \lambda_k \left(  \int_0^{\delta} \dfrac{ e^{-\lambda_k t}}{|c(t)|}\, dt \right)^2,\;\; m \ge 2. 
\end{equation}   
To bound \eqref{sqq} from below, note that 
 $$
|c(t)| = \left(\sum_{k=1}^{m} \dfrac{1 - e^{-2\lambda_k t}}{2\lambda_k}\right)^{1/2} \leq 
%{c}_{\infty}(t) := 
 \left(\sum_{k=1}^{+\infty} \dfrac{1 - e^{-2\lambda_k t}}{2\lambda_k}\right)^{1/2} = \left(\int_0^t \left[\sum_{k=1}^{+\infty} e^{-2\lambda_k s}\right] \ud s\right)^{1/2}
$$
holds for any $t \geq 0$. Now, if there is a positive constant $c_0$ such that $\lambda_k \geq c_0 k^2$ for any $k \ge 1$, then
$$
\sum_{k=1}^{+\infty} e^{-2\lambda_k s} \leq \sum_{k=1}^{+\infty} e^{-2{c_0} k^2 s} \leq \int_0^{+\infty} e^{-2{c_0} z^2 s} \ud z = \sqrt{\frac{\pi}{2{c_0}s}},\;\; s>0,
$$ 
yielding    
$$
 |c(t)|\leq \left(\int_0^t \sqrt{\frac{\pi}{2{c_0}s}} \ud s\right)^{1/2} = \left(\frac{2\pi t}{{c_0}}\right)^{1/4}.  
$$
Up to now we have found that
\begin{gather*}
A_m \ge \frac{{1}}{{\pi}}
  \sum_{k=1}^m    \lambda_k \left(  \int_0^{\delta} { e^{-\lambda_k t}}  \left(\frac{{c_0}}{2\pi t}\right)^{1/4} \, dt \right)^2,\;\; m \ge 2.  
\end{gather*}
Now, exploit that 
$$
\int_0^{\delta} t^{-\frac 14} e^{-\lambda t}dt = \left(\frac{1}{\lambda}\right)^{\frac 34} \int_0^{\lambda \delta} s^{-\frac 14} e^{- s} ds 
\geq \left(\frac{1}{\lambda}\right)^{\frac 34} \int_0^{{c_0}\delta} s^{-\frac 14} e^{-s}\ud s
$$  
holds for every $\lambda \geq {c_0}$, to get (after recalling that, in particular, $\lambda_k \ge c_0$, for any $k \ge 1$) 
\begin{align*}
A_m &\geq \frac{1}{\pi} \sqrt{\frac{{c_0}}{2\pi}}\, \sum_{k=1}^m \lambda_k \left(\int_0^{\delta} t^{-\frac 14} e^{-\lambda_k t}\ud t\right)^2  
\geq \frac{1}{\pi} \sqrt{\frac{{c_0}}{2\pi}}\sum_{k=1}^m \lambda_k  \left(\frac{1}{\lambda_k}\right)^{\frac 32} \left(\int_0^{{c_0}\delta} s^{-\frac 14} e^{-s}\ud s\right)^2 \\
&= \frac{1}{\pi} \sqrt{\frac{{c_0}}{2\pi}} \left(  \int_0^{{c_0}\delta} s^{-\frac 14} e^{-s}\ud s\right)^2 \sum_{k=1}^m \frac{1}{\sqrt{\lambda_k}}\ .
\end{align*}
Thus, if $\lambda_k \sim k^2$ as $k \to +\infty$, then $\sum_{k=1}^m \frac{1}{\sqrt{\lambda_k}} \sim \log m$ as $m \to +\infty$, and \eqref{qee} holds. This finishes the proof.


\begin{thebibliography}{999}  

\bibitem{bass}
 S. R. Athreya, R. F. Bass, E. A. Perkins, H\"older norm estimates for elliptic operators on finite and infinite-dimensional spaces. Trans. Amer. Math. Soc. 357 (2005)  5001-5029.
 
%\bibitem{casarino}  ?? V. Casarino, P.  Ciatti, P.  Sj\"ogren, 
%Riesz transforms of a general Ornstein-Uhlenbeck semigroup.  
%Calc. Var. Partial Differential Equations 60 (2021), no. 4, Paper No. 135, 32 pp. 
% weak type (1,1) estimates  of the Gaussian Riesz transforms with respect to the invariant measure 
 
\bibitem{casar} V.  Casarino, P.  Ciatti, P. Sjogren, 
On the maximal operator of a general Ornstein-Uhlenbeck semigroup. 
Math. Z. 301 (2022)  2393-2413.  
 

\bibitem{goldysChojnowskaJFA} A. Chojnowska-Michalik, B.  Goldys,  
Generalized Ornstein-Uhlenbeck semigroups: Littlewood-Paley-Stein inequalities and the P.A. Meyer equivalence of norms, \!\!   J. Funct. Anal. 182 (2001) 243-279.
 

\bibitem{ChGold02} A. Chojnowska-Michalik,  B. Goldys,  Symmetric Ornstein-Uhlenbeck semigroups and their generators. Probab. Theory Related Fields 124 (2002)  459-486.
  

%\bibitem{dapratoGoldys}  ??? G. Da Prato and B. Goldys, On perturbations of symmetric Gaussian diffusions, Stochastic Anal. Appl. 17 (1999), 369-382. 



\bibitem{DL}  G.  Da Prato, A.  Lunardi, 
 {On the Ornstein-Uhlenbeck operator
 in spaces of continuous functions}, {J. Funct.
 Anal.} {  131} (1995) 94-114. 


\bibitem{dapratoLincei} G. Da Prato, Characterization of the domain of an elliptic operator of infinitely many variables in $L^2(\mu)$-spaces, Atti Accad. Naz. Lincei Cl. Sci. Fis. Mat. Natur. Rend. Lincei (9) Mat. Appl. 8 (1997) 101-105. 
  
 
 \bibitem{DZ1} G. Da Prato, J.
 Zabczyk, Second Order Partial Differential
Equations in Hilbert Spaces, London Math. Soc. Lecture Notes
vol. 293, Cambridge University Press,  2002.
 
 \bibitem{DZ} G. Da Prato, J. Zabczyk, Stochastic equations in 
 infinite
dimensions, Encyclopedia of Mathematics and its Applications, \! II edition, 152. Cambridge University Press,  2014.  

 \bibitem{D2} G. Da Prato, A new regularity result for Ornstein-Uhlenbeck generators and applications, J. Evol. Eq. 3 (2003) 485-498. 
 
 
\bibitem{DFPR}   G. Da Prato, F. Flandoli, E. Priola, M. R\"ockner,
     Strong uniqueness for stochastic evolution equations
in Hilbert spaces  perturbed by a bounded measurable drift, {Ann. of Prob.}    41
(2013)  3306-3344.   
 
\bibitem{fleming} W. Fleming,  Functions of several variables. Second edition. Undergraduate Texts in Mathematics. Springer-Verlag, New York-Heidelberg, 1977.  
 
 
\bibitem{gutierrez} C. E. Gutierrez,   C. Segovia, and J. L. Torrea, On higher Riesz transforms for Gaussian measures, J. Fourier Anal. Appl. 2 (1996) 583-596.  


\bibitem{hairer} M. Hairer, An Introduction to Stochastic PDEs,  2009,  available at \url{arXiv:0907.4178v1.}   


\bibitem{Nu}  D. Nualart,   ``The Malliavin calculus and related topics.
 Probability and its Applications'',   Springer-Verlag, New York, 1995.

 
\bibitem{lunardi}  A. Lunardi,  On the Ornstein-Uhlenbeck operator in $L^2$ spaces with respect to invariant measures. Trans. Amer. Math. Soc. 349 (1997) 155-169. 
 
 
\bibitem{maasNeerven09} J. Maas, J. van Neerven, Boundedness of Riesz transforms for elliptic operators on abstract Wiener spaces. J. Funct. Anal. 257 (2009)  2410-2475.


\bibitem{metafAltri} G. Metafune, J.  Pruss, A.  Rhandi, R.  Schnaubelt, The domain of the Ornstein-Uhlenbeck operator on an $L_p$-space with invariant measure. Ann. Sc. Norm. Super. Pisa Cl. Sci. (5) 1 (2002), no. 2, 471-485.


\bibitem{meyer}  P. A. Meyer, Note sur les processus d'Ornstein-Uhlenbeck, in ``Seminaire de Probabilites XVI,'' Lecture Notes in Math., Vol. 920, pp. 95-133, Springer-Verlag, Berlin/New York, 1982.
 
\bibitem{meyer84} 
 P. A. Meyer. Transformations de Riesz pour les Lois Gaussiennes. Seminaire de probabilites de Strasbourg 18:179-193, 1984.
 

\bibitem{priolaAOP}  E. Priola, An optimal regularity result for Kolmogorov equations and weak uniqueness for some
critical SPDEs, Ann. of Prob. 49 (2021) 1310-1346. 
%Correction in 
% arxiv....

 
\bibitem {priolarxiv} E. Priola, Correction to {\it ``An optimal regularity result for
Kolmogorov equations and weak uniqueness for
some critical SPDEs''}  arxiv...
 


\bibitem{Shigekawa92} I. Shigekawa, Sobolev spaces over the Wiener space based on an Ornstein-Uhlenbeck operator, J. Math. Kyoto Univ. 32 (1992) 731-748.

\bibitem{ShigekawaYosida} I. Shigekawa and N. Yoshida, Littlewood-Paley-Stein inequality for a symmetric diffusion, J. Math. Soc. Japan 44 (1992) 249-280.




\bibitem{spiegel} M. R. Spiegel, Mathematical Handbook of Formulas and Tables, Schaum, 1968. 
 
\bibitem{stein} E. Stein, ``Topics in Harmonic Analysis,'' Princeton Univ. Press, Princeton, NJ, 1970.



  
 
  
\end{thebibliography}
\end{document}